\documentclass[reqno]{amsart}
\usepackage[utf8]{inputenc}

\usepackage{amssymb}
\usepackage{graphicx}
\usepackage{latexsym}
\usepackage{stmaryrd}
\usepackage{enumitem}
\usepackage[bookmarks]{hyperref}
\usepackage{multirow}
\usepackage{xcolor}
\usepackage{relsize}
\usepackage{comment}
\usepackage{xifthen}
\usepackage{mathrsfs}
\usepackage{svg}
\usepackage{mathtools}
\colorlet{darkishRed}{red!80!black}
\colorlet{darkishBlue}{blue!60!black}
\colorlet{darkishGreen}{green!60!black}
\hypersetup{
    draft = false,
    bookmarksopen=true,
    colorlinks,
    linkcolor={red!60!black},
    citecolor={green!60!black},
    urlcolor={blue!60!black}
}

\usepackage{geometry}
\newcommand{\sun}{sun}

\def\lowfwd #1#2#3{{\mathop{\kern0pt #1}\limits^{\kern#2pt\raise.#3ex
\vbox to 0pt{\hbox{$\scriptscriptstyle\rightarrow$}\vss}}}}
\def\lowbkwd #1#2#3{{\mathop{\kern0pt #1}\limits^{\kern#2pt\raise.#3ex
\vbox to 0pt{\hbox{$\scriptscriptstyle\leftarrow$}\vss}}}}

\def\ve{\kern-1.5pt\lowfwd e{1.5}2\kern-1pt}
\def\ev{\kern-1pt\lowbkwd e{0.5}2\kern-1pt}
\def\vf{\kern-2pt\lowfwd f{2.5}2\kern-1pt}

\renewcommand{\subset}{\subseteq}
\renewcommand{\supset}{\supseteq}
\newcommand{\at}{attached to~}

\newcommand{\Abs}[1]{\partial_{\Omega} {#1}}

\newcommand{ \N } { \mathbb{N} }

\usepackage{caption}
\usepackage{subcaption}

\makeatletter

\def\calCommandfactory#1{%
   \expandafter\def\csname c#1\endcsname{\mathcal{#1}}}
\def\frakCommandfactory#1{%
   \expandafter\def\csname frak#1\endcsname{\mathfrak{#1}}}

\newcounter{ctr}
\loop
  \stepcounter{ctr}
  \edef\X{\@Alph\c@ctr}
  \expandafter\calCommandfactory\X
  \expandafter\frakCommandfactory\X
  \edef\Y{\@alph\c@ctr}
  \expandafter\frakCommandfactory\Y
\ifnum\thectr<26
\repeat

\renewcommand{\cC}{\mathscr{C}}

\newcommand{\rsep}[2]{({#1},{#2})}
\newcommand{\lsep}[2]{({#1},{#2})}
\newcommand{\sep}[2]{\{{#1},{#2}\}}

\newcommand{\dc}[1]{\lceil #1\rceil}
\newcommand{\uc}[1]{\lfloor #1\rfloor}
\newcommand{\guc}[1]{\lfloor\mkern-1.4\thinmuskip\lfloor #1\rfloor\mkern-1.4\thinmuskip\rfloor}

\newtheorem{theorem}{Theorem}[section]

\newtheorem{lemma}[theorem]{Lemma}
\newtheorem{example}[theorem]{Example}

\newtheorem{innercustomthm}{Theorem}

\newenvironment{customthm}[1]
  {\renewcommand\theinnercustomthm{#1}\innercustomthm}
  {\endinnercustomthm}

\theoremstyle{definition}

\newtheorem{definition}[theorem]{Definition}

\theoremstyle{remark}

\setenumerate{label={\normalfont (\roman*)},itemsep=0pt}

\title[Countably determined ends and graphs]{\Large Countably determined ends and graphs}

\author{Jan Kurkofka}
\author{Ruben Melcher}
\address{University of Hamburg, Department of Mathematics, Bundesstraße 55 (Geomatikum), 20146 Hamburg, Germany}
\email{\{jan.kurkofka, ruben.melcher\} @uni-hamburg.de}
\keywords{infinite graph; countably determined; end; direction; end space; axioms of countability; first countable; second countable; normal tree; tree-decomposition}

\@namedef{subjclassname@2020}{\textup{2020} Mathematics Subject Classification}
\subjclass[2020]{05C63, 05C75, 54D40, 54D70, 54D65}

\begin{document}

\begin{abstract}
The \emph{directions} of an infinite graph $G$ are a tangle-like description of its ends: they are choice functions that choose compatibly for all finite vertex sets $X\subset V(G)$ a component of $G-X$.

Although every direction is induced by a ray, there exist directions of graphs that are not uniquely determined by any countable subset of their choices.
We characterise these directions and their countably determined counterparts in terms of star-like substructures or rays of the graph.

Curiously, there exist graphs whose directions are all countably determined but which cannot be distinguished all at once by countably many choices.

We structurally characterise the graphs whose directions can be distinguished all at once by countably many choices, and we structurally characterise  the graphs which admit no such countably many choices. Our characterisations are phrased in terms of normal trees  and tree-decompositions.

Our four (sub)structural characterisations imply combinatorial characterisations of the four classes of infinite graphs that are defined by the first and second axiom of countability applied to their end spaces: the two classes of graphs whose end spaces are first countable or second countable, respectively, and the complements of these two classes.

\end{abstract}
\vspace*{-1cm}
\maketitle

\vspace*{-.7cm}
\section{Introduction}

\noindent Halin~\cite{Halin_Enden64} defined the ends of an infinite graph `from below' as  equivalence classes of rays in the graph, where two rays are equivalent if no finite set of vertices separates them.
As a complementary description of Halin's ends, Diestel and Kühn~\cite{diestelKuhn2003directions} introduced the notion of directions of infinite graphs.
These are defined `from above': 
A \emph{direction} of a graph $G$ is a map $f$, with domain the collection $\cX=\cX(G)$ of all finite vertex sets of $G$, that assigns to every $X\in\cX$ a component $f(X)$ of $G-X$ such that $f(X)\supset f(X')$ whenever~$X\subset X'$.

Every end $\omega$ of $G$ defines a direction $f_\omega$ of $G$ by letting $f_\omega(X)$ be the component of $G-X$ that contains a subray of every ray in $\omega$.
Diestel and Kühn showed that the natural map $\omega\mapsto f_\omega$ is in fact a bijection between the ends of $G$ and its directions.
This correspondence is now well known and has become a standard tool in the study of infinite graphs.
See~\cite{StarComb1StarsAndCombs,EndsAndTangles,VTopComp,ApproximatingNormalTrees,EndsTanglesCrit,StoneCechTangles} for examples.

The domain of the directions of $G$ might be arbitrarily large as its size is equal to the order of $G$.
This contrasts with the fact that every direction of $G$ is induced by a ray of $G$ and rays have countable order. 
Hence the question arises whether every direction  of $G$ is `countably determined' in $G$ also by a countable subset of its choices. 
A \emph{directional choice} in $G$ is a pair $(X,C)$ of a finite vertex set $X\in\cX$ and a component $C$ of $G-X$. We say that a directional choice $(X,C)$ in $G$ \emph{distinguishes} a direction $f$ from another direction~$h$ if $f(X)=C$ and $h(X)\neq C$. A direction $f$ of $G$ is \emph{countably determined} in $G$ if there is a countable set of directional choices in $G$ that distinguish $f$ from every other direction of $G$.

Curiously, the answer to this question is in the negative:
Consider the graph $G$ that arises from the uncountable complete graph $K^{\aleph_1}$ by adding a new ray $R_v$ for every vertex $v\in K^{\aleph_1}$ so that $R_v$ meets $K^{\aleph_1}$ precisely in its first vertex $v$ and $R_v$ is disjoint from all the other new rays~$R_{v'}$.
Then $K^{\aleph_1}\subset G$ induces a direction of $G$ that is not countably determined in $G$.

This example raises the question of which directions of a given graph $G$ are countably determined.
In the first half of our paper we answer this question: we characterise for every graph $G$, by unavoidable substructures, both the countably determined directions of $G$ and its directions that are not countably determined.

If $R\subset G$ is any ray, then every finite initial segment $X$ of $R$ naturally defines a directional choice in~$G$, namely $(X,C)$ for the component $C$ that contains $R-X$. 
Let us call  $R$ \emph{directional} in $G$ if its induced direction is distinguished from every other direction of $G$ by the directional choices that are defined by~$R$.
By definition, every direction of $G$ that is induced by a directional ray is  countably determined in~$G$.
Surprisingly, our characterisation implies that the converse holds as well: if a direction of $G$ is distinguished from every other direction by countably many directional choices $(X,C)$, then no matter how the vertex sets $X$ lie in $G$ we can always assume that the sets $X$ are the finite initial segments of a directional ray.

\begin{customthm}{1}\label{introTheoremOne}
For every graph $G$ and every direction $f$ of $G$ the following assertions are equivalent:
\begin{enumerate}
    \item The direction $f$ is countably determined in $G$.
    \item The direction $f$ is induced by a directional ray of $G$.
    \end{enumerate}
\end{customthm}

As our second main result we characterise by unavoidable substructures the directions of any given graph that are not countably determined in that graph, and thereby complement our first characterisation.
Our theorem is phrased in terms of substructures that are uncountable star-like combinations either of rays or of double rays.
Recall that a vertex $v$ of a graph $G$ \emph{dominates} a ray $R\subset G$ if there is an infinite $v$--$R$ fan in~$G$. 
An end of $G$ is \emph{dominated} if one (equivalently:~each) of its rays is dominated, see~\cite{diestel2015book}.
Given a direction $f$ of $G$ we write $\omega_f$ for the unique end $\omega$ of $G$ whose rays induce~$f$, i.e., which satisfies $f_\omega = f$.
If $G$ is a graph and $(T,\cV)$ is a tree-decomposition of $G$ that has finite adhesion, then every direction of~$G$ either corresponds to a direction of $T$ or lives in a part of $(T,\cV)$; see Section~\ref{subsec:tdcsStreesEnds}. An \emph{uncountable star-decomposition}  is a tree-decomposition whose decomposition tree is a star $K_{1,\kappa}$ for some uncountable cardinal~$\kappa$.

\begin{customthm}{2}\label{introTheoremTwo}
For every graph $G$ and every direction $f$ of $G$ the following assertions are equivalent:
\begin{enumerate}
    \item The direction $f$ is not countably determined in $G$.
    \item The graph $G$ contains either
    \begin{itemize}[label=\textbf{--}]
        \item uncountably many disjoint pairwise inequivalent rays all of which start at vertices that\\dominate~$\omega_f$, or
        \item uncountably many disjoint double rays, all having one tail in~$\omega_f$ and another not in $\omega_f$,\\so that the latter tails are inequivalent for distinct double rays.
    \end{itemize}
\end{enumerate}
Moreover, if \emph{(ii)} holds, we can find the (double) rays together with an uncountable star-decomposition of $G$ of finite adhesion such that $f$ lives in the central part and each (double) ray has a tail in its own leaf part.
\end{customthm}

\noindent Note that (ii) clearly implies (i).

Does the local property that every direction of $G$ is countably determined in $G$ imply the stronger global property that there is one countable set of directional choices that distinguish every two directions of $G$ from each other?
We answer this question in the negative; see Example~\ref{countablyDeterminedLocalNotGlobal}. Let us call a graph $G$ \emph{countably determined}  if there is a countable set of directional choices in $G$ that distinguish every two directions of $G$ from each other.

In the second half of our paper we structurally characterise both the graphs that are countably determined and the graphs that are not countably determined.
A~rooted tree $T\subset G$ is \emph{normal} in $G$ if the endvertices of every $T$-path  in $G$ are comparable in the tree-order of~$T$, cf.~\cite{diestel2015book}. (A \emph{$T$-path} in $G$ is a non-trivial path that meets $T$ exactly in its endvertices.)

\newpage

\begin{customthm}{3}\label{introTheoremThree}
For every connected graph $G$ the following assertions are equivalent:
\begin{enumerate}
    \item $G$ is countably determined.
    \item  $G$ contains  a countable normal tree  that contains a ray from every end of~$G$.
    \end{enumerate}
\end{customthm}

Complementing this characterisation we structurally characterise, as our fourth main result, the graphs that are not countably determined.

\begin{customthm}{4}\label{introTheoremFour}
For every connected graph $G$ the following assertions are equivalent:
\begin{enumerate}
    \item   $G$ is not countably determined.
    \item   $G$ has an uncountable star-decomposition of finite adhesion such that in every leaf part  there lives a direction of $G$.
    \end{enumerate}
\end{customthm}

Interestingly, countably determined directions and countably determined graphs admit natural topological interpretations.
Over the course of the last two decades, the topological properties of end spaces have been extensively investigated, see e.g.~\cite{diestel2006end,diestelKuhn2003directions,polat1996ends,polat1996ends2,sprussel2008end}.
However, not much is known about such fundamental properties as countability axioms. 
Recall that a topological space is \emph{first countable} at a given point if it has a countable neighbourhood base at that point. 
A direction of a graph $G$ is countably determined in $G$ if and only if it is defined by an end that has a countable neighbourhood base in the end space of~$G$ (Theorem~\ref{thm: equiv. first countable}). 
Thus, Theorems~\ref{introTheoremOne} and~\ref{introTheoremTwo} characterise combinatorially when the end space of a graph is first countable or not first countable at a given end, respectively.
Similarly, a graph is countably determined if and only if its end space is \emph{second countable} in that its entire topology has a countable base (Theorem~\ref{lemma: G count. det. iff Omega second count.}).  
Therefore, Theorems~\ref{introTheoremThree} and~\ref{introTheoremFour} characterise combinatorially the  infinite graphs whose end spaces are second countable or not second countable, respectively.
Furthermore, our four theorems imply similar results for the space $\vert G\vert$ formed by a graph $G$ together with its end space; see Section~\ref{section: modG}.

This paper is organised as follows:
In the next section we give a reminder on end spaces and recall all the results from graph theory and general topology that we need. 
We prove Theorems~\ref{introTheoremOne} and~\ref{introTheoremTwo} in Section~\ref{section: first coutable} and we prove Theorems~\ref{introTheoremThree} and~\ref{introTheoremFour} in Section~\ref{section: second countable}. Finally, in Section~\ref{section: modG} we consider the spaces $\vert G \vert$.

\section{preliminaries}\label{setion: preliminaries}

\noindent For graph theoretic terms we follow the terminology in \cite{diestel2015book}.
For topological notions we follow the terminology in \cite{EngelkingBook}. 

\subsection{Ends of graphs}
A $1$-way infinite path is called a \emph{ray} and the subrays of a ray are its \emph{tails}. Two rays in a graph $G = (V,E)$ are \emph{equivalent} if no finite set of vertices separates them; the corresponding equivalence classes of rays are the \emph{ends} of $G$. The set of ends of a graph $G$ is denoted by $\Omega = \Omega(G)$. 
If~$X \subseteq V$ is finite and $\omega \in \Omega$, 
there is a unique component of $G-X$ that contains a tail of every ray in~$\omega$, which we denote by $C(X,\omega)$. 
If $C$ is any component of $G-X$, we write $\Omega(X,C)$ for the set of ends $\omega$ of~$G$ with $C(X,\omega) = C$, and abbreviate $\Omega(X,\omega) := \Omega(X,C(X,\omega))$. The ends in $\Omega(X,\omega)$ are said to \emph{live} in $C(X,\omega)$. For two ends $\omega_1$ and $\omega_2$ a finite vertex set $X \subseteq V$ \emph{separates} $\omega_1$ and $\omega_2$ if they live in distinct components of $G-X$.

The collection of sets $\Omega(X,C)$ with $X\subset V$ finite and $C$ a component of $G-X$ form a basis for a topology on $\Omega$.
This topology is Hausdorff, and it
has a basis consisting of closed-and-open sets. 
The space $\Omega(G)$ with this topology is called the \emph{end space} of $G$. 

A vertex $v$ of $G$ \emph{dominates} a ray $R\subset G$ if there is an infinite $v$--$R$ fan in~$G$. 
An end of $G$ is \emph{dominated} by $v$ if one (equivalently:~each) of its rays is dominated by~$v$.
If a vertex $v$ of $G$ dominates an end $\omega$ of $G$, then $v \in  C(X, \omega)$ for all finite sets $X \subseteq V(G)$ with $v \not\in X$.

Recall that a \emph{comb} is the union of a ray $R$ (the comb's \emph{spine}) with infinitely many disjoint finite paths, possibly trivial, that have precisely their first vertex on~$R$. 
The last vertices of those paths are the \emph{teeth} of this comb.
Given a vertex set $U$, a \emph{comb attached to} $U$ is a comb with all its teeth in $U$, and a \emph{star attached to} $U$ is a subdivided infinite star with all its leaves in $U$.
Then the set of teeth is the \emph{attachment set} of the comb, and the set of leaves is the \emph{attachment set} of the star.
The following lemma is~\cite[Lemma~8.2.2]{diestel2015book}:

\begin{lemma}[Star-Comb Lemma]\label{lemma: star-comb}
Let $U$ be an infinite set of vertices in a connected graph $G$.
Then $G$ contains either a comb 
\at $U$ or a star \at $U$.
\end{lemma}

Let us say that an end $\omega$ of $G$ is contained \emph{in the closure} of $M$, where $M$ is either a subgraph of $G$ or a set of vertices of $G$, if for every $X\in\cX$ the component $C(X,\omega)$ meets $M$.
Equivalently, $\omega$ lies in the closure of $M$ if and only if $G$ contains a comb \at $M$ with its spine in $\omega$.
We write $\Abs{M}$ 
for the subset of $\Omega$ that consists of the ends of $G$ lying in the closure of $M$. 
A vertex set $U\subset V(G)$ is \emph{dispersed} in~$G$ if $\Abs{U}=\emptyset$.

\subsection{Normal trees}

The \emph{tree-order} of a rooted tree $T=(T,r)$ is defined by setting $u \leq v$ if $u$ lies on the unique path $rTv$ from $r$ to $v$ in $T$. Given $n \in \N$, the \emph{$n$th level} of $T$ is the set of vertices at distance $n$ from $r$ in $T$. The \emph{down-closure} of a vertex $v$ is the set $\lceil v \rceil := \{\,u \colon u \leq v\,\}$; its \emph{up-closure} is the set $\lfloor v \rfloor := \{\,w \colon v \leq w\,\}$. The down-closure of $v$ is always a finite chain, the vertex set of the path $rTv$. A ray $R \subset T$ starting at the root is called a \emph{normal ray} of $T$.

A rooted tree $T$ contained in a graph $G$ is \emph{normal} in~$G$ if the endvertices of every \mbox{$T$-path} in $G$ are comparable in the tree-order of~$T$.
Here, for a given graph~$H$, a path $P$ is said to be an $H$-\emph{path} if $P$ is non-trivial and meets $H$ exactly in its endvertices.
We remark that for a normal tree $T\subset G$ the neighbourhood $N(C)$ of every component $C$ of $G-T$ forms a chain in $T$. 

The \emph{generalised up-closure} $\guc{x}$ of a vertex $x\in T$ is the union of $\uc{x}$ with the vertex set of $\bigcup \cC(x)$, where the set $\cC(x)$ consists of those components of $G-T$ whose neighbourhoods meet $\uc{x}$.
Every graph $G$ reflects the separation properties of each normal tree $T\subset G$: 

\begin{lemma}[{\cite[Lemma~2.10]{StarComb1StarsAndCombs}}]\label{lemma: NTseparationAndComponents}
Let $G$ be any graph and let $T\subset G$ be any normal tree.
\begin{enumerate}
    \item Any two vertices $x,y\in T$ are separated in $G$ by the vertex set $\dc{x}\cap\dc{y}$.
    \item Let $W\subset V(T)$ be down-closed. Then the components of $G-W$ come in two types: the components that avoid $T$; and the components that meet $T$, which are spanned by the 
    sets $\guc{x}$ with $x$ minimal in $T-W$. 
\end{enumerate}
\end{lemma}

As a consequence, the normal rays of a normal spanning tree $T\subset G$ reflect the end structure of $G$ in that every end of $G$ contains exactly one normal ray of $T$, \cite[Lemma~8.2.3]{diestel2015book}.
More generally:

\begin{lemma}[{\cite[Lemma~2.11]{StarComb1StarsAndCombs}}]\label{lemma: NormalTreeNormalRay}
If $G$ is any graph and $T\subset G$ is any normal tree,
then every end of $G$ in the closure of $T$ contains exactly one normal ray of $T$.
Moreover, sending these ends to the normal rays they contain defines a bijection between $\Abs{T}$ and the normal rays of~$T$.
\end{lemma}

Not every connected graph has a normal spanning tree. However, every countable connected graph does. More generally:

\begin{lemma}[{Jung~\cite{jung1969wurzelbaume}, \cite[Corollary~3.3]{StarComb1StarsAndCombs}}]\label{lemma: JungCountableSet}
Let $G$ be any graph and let $U\subset V(G)$ be any vertex set. 
If~$U$ is countable and $v$ is any vertex of $G$, then $G$ contains a normal tree that contains $U$ cofinally and is rooted in~$v$.
\end{lemma}

If $H$ is a subgraph of $G$, then rays equivalent in $H$ remain equivalent in $G$; in other words, every end of $H$ can be interpreted as a subset of an end of $G$, so the natural inclusion map $\iota \colon \Omega(H) \to \Omega(G)$ is well-defined. A subgraph $H \subset G$ is \emph{end-faithful} if this inclusion map $\iota$ is a bijection. The terms \emph{end-injective} and \emph{end-surjective} are defined accordingly. Normal trees are always end-injective; hence, normal trees are end-faithful as soon as they are end-surjective.

\subsection{Tree-decompositions, separations and ends}\label{subsec:tdcsStreesEnds}

Recall from~\cite[Section~12.5]{diestel2015book} that a \emph{tree-decomposition} of a graph $G$ is an ordered pair $(T,\cV)$ of a tree $T$ and a family $\cV=(V_t)_{t\in V(T)}$ of \emph{parts} $V_t\subset V(G)$ such that:
\begin{enumerate}
    \item $V(G)=\bigcup_{t\in T}V_t$;
    \item every edge of $G$ has both endvertices in $V_t$ for some~$t$;
    \item $V_{t_1}\cap V_{t_3}\subset V_{t_2}$ whenever $t_2\in t_1 T t_3$.
\end{enumerate}

When we introduce a tree-de\-com\-po\-si\-tion as $(T,\cV)$ we tacitly assume $\cV=(V_t)_{t\in T}$.
Every edge of $T$ induces a `separation' of $G$, as follows.

A \emph{separation} of a graph $G$ is an unordered pair $\{A,B\}$ with $A\cup B=V(G)$ and such that no edge of $G$ `jumps' the \emph{separator} $A\cap B$, meaning that no edge of $G$ runs between $A\setminus B$ and $B\setminus A$.
Every edge $t_1 t_2\in T$ \emph{induces} the separation $\{U_1,U_2\}$ of $G$ that is defined by $U_1:=\bigcup_{t\in T_1}V_t$ and $U_2:=\bigcup_{t\in T_2}V_t$, where $T_1$ and $T_2$ are the components of $T-t_1 t_2$ containing $t_1$ and $t_2$ respectively.
Then the separator $U_1\cap U_2=V_{t_1}\cap V_{t_2}$ is an \emph{adhesion set} of~$(T,\cV)$.
We usually refer to the adhesion sets as separators.
The tree-decomposition $(T,\cV)$ has \emph{finite adhesion} if all its adhesion sets are finite.

Suppose now that $(T,\cV)$ has finite adhesion.
Then every end $\omega$ of $G$ either `lives' in a unique part $V_t$ or `corresponds' to a unique end $\eta$ of~$T$.
Here, $\omega$ \emph{lives} in $V_t$ if some (equivalently:\ every) ray in $\omega$ has infinitely many vertices in $V_t$.
And $\omega$ \emph{corresponds} to $\eta$ if some (equivalently:\ every) ray $R\in\omega$ follows the course of some (equivalently:\ every) ray $S\in\eta$ (in that for every tail $S'\subset S$ the ray $R$ has infinitely many vertices in $\bigcup_{t\in S'}V_t$).
If $f$ is a direction of $G$, then $f$ \emph{lives} in the part of $(T,\cV)$ or \emph{corresponds} to the end of $T$ that $\omega_f$ lives in or corresponds to, respectively.

If $(T,\cV)$ is a star-decomposition and $s$ is the centre of the star~$T$, then the induced separations of $(T,\cV)$ form a `star of separations', as follows.
Every separation $\{A,B\}$ has two \emph{orientations}, the \emph{oriented separations} $(A,B)$ and $(B,A)$.
A partial ordering ${\le}$ is defined on the set of all oriented separations of $G$ by letting
\begin{align*}
    (A,B)\le (C,D)\;:\Leftrightarrow\;A\subset C\text{ and }B\supset D.
\end{align*}
A \emph{star} (\emph{of separations}) is a set $\{\,(A^i,B^i)\colon i\in I\,\}$ of oriented separations $(A^i,B^i)$ of~$G$ such that $(A^i,B^i)\le (B^j,A^j)$ for every two distinct indices~$i,j\in I$.
If we write the edge set of $T$ as $\{t_i s\colon i\in I\}$, then the induced separations $\{U_1^i,U_2^i\}$ defined by $U_1^i:=V_{t_i}$ and $U_2^i:=\bigcup_{t\in T-t_i}V_t$ form a star of separations $\{\,(U_1^i,U_2^i)\colon i\in I\,\}$.
Conversely, every star of separations $\{\,(A^i,B^i)\colon i\in I\,\}$ defines a star-decomposition with leaf parts $A^i$ ($i\in I$) and central part $\bigcap_{i\in I}B^i$.
If this star is $\{\,(U_1^i,U_2^i)\colon i\in I\,\}$, we retrieve $(T,\cV)$.

The following non-standard notation often will be useful as an alternative perspective on separations of graphs.
For a vertex set $X\subset V(G)$ we denote the collection of the components of $G-X$ by $\cC_X$.
If any $X\subset V(G)$ and $\cC\subset\cC_X$ are given, then these give rise to a separation of $G$ which we denote by
\begin{align*} 
    \sep{X}{\cC}:=\big\{\;V\setminus V[\cC]\;,\;X\cup V[\cC]\;\big\}
\end{align*}
where $V[\cC]=\bigcup\,\{\,V(C)\mid C\in\cC\,\}$.
Note that every separation $\{A,B\}$ of $G$ can be written in this way. 
For the orientations of $\sep{X}{\cC}$ we write
\begin{align*}
\rsep{X}{\cC}:=\big(\;V\setminus V[\cC]\;,\;X\cup V[\cC]\;\big)\quad\text{and}\quad\lsep{\cC}{X}:=\big(\;V[\cC]\cup X\;,\;V\setminus V[\cC]\;\big).
\end{align*}
If $C$ is a component of $G-X$ we write $\sep{X}{C}$ instead of $\sep{X}{\{C\}}$.  Similarly, we write $\rsep{X}{C}$ and $\lsep{C}{X}$ instead of $\rsep{X}{\{C\}}$ and $\lsep{\{C\}}{X}$, respectively.

We conclude this section with a paragraph on how tree-decompositions can be used to distinguish ends of graphs.
The \emph{order} of a separation is the cardinality of its separator.
A finite-order separation $\{A,B\}$ of $G$ \emph{distinguishes} two ends $\omega_1$ and $\omega_2$ of $G$ if $C(A\cap B,\omega_1)\subset G[A\setminus B]$ and $C(A\cap B,\omega_2)\subset G[B\setminus A]$ (or vice versa).
If $\{A,B\}$ distinguishes $\omega_1$ and $\omega_2$ and has minimal order among all the separations of $G$ that distinguish $\omega_1$ and $\omega_2$, then $\{A,B\}$ distinguishes $\omega_1$ and $\omega_2$ \emph{efficiently}.
The ends $\omega_1$ and $\omega_2$ are said to be $k$-\emph{distinguishable} for an integer $k\ge 0$ if there is a separation of $G$ of order at most $k$ that distinguishes $\omega_1$ and $\omega_2$.
We say that $(T,\cV)$ \emph{distinguishes} $\omega_1$ and $\omega_2$ if some edge of $T$ induces a separation of $G$ that distinguishes $\omega_1$ and~$\omega_2$; it distinguishes $\omega_1$ and $\omega_2$ \emph{efficiently} if the induced separation can be chosen to distinguish $\omega_1$ and $\omega_2$ efficiently.
The following theorem is an immediate consequence of~\cite[Corollary~6.6]{CanonicalTreeOfTangles} and the construction in~\cite[Theorem~6.2]{CanonicalTreeOfTangles} that leads to the corollary.
(The construction in the proof of Theorem~6.2 ensures that the tree of tree-decompositions obtained in Corollary~6.6 has finite height, and  it is straightforward to combine all the tree-decompositions into a single one.)

\begin{theorem}\label{magicTDC}
Every connected graph $G$ has for every number $k\in\N$ a tree-decomposition that efficiently distinguishes all the $k$-distinguishable ends of~$G$.
\end{theorem}


\section{Countably determined directions and the first axiom\texorpdfstring{\\}{ }of countability}\label{section: first coutable}

\noindent In this section we characterise for every graph $G$, by unavoidable substructures, both the countably determined directions of $G$ and its directions that are not countably determined.

Given a graph $G$ we call a ray $R\subset G$ \emph{topological} in $G$ if the end $\omega$ of $G$ that contains $R$ has a countable neighbourhood base $\{\,\Omega(X_n,\omega)\colon n\in\N\,\}$ in $\Omega(G)$ where each vertex set $X_n$ consists of the first $n$ vertices of~$R$.
Our first lemma shows that rays are directional if and only if they are topological:

\begin{lemma}\label{directionalTopological}
For every graph $G$ and every ray $R\subset G$ the following assertions are equivalent:
\begin{enumerate}
    \item $R$ is directional in $G$.
    \item $R$ is topological in $G$.
\end{enumerate}
\end{lemma}
\begin{proof}
Let us write $\omega$ for the end of $G$ that is represented by $R$, and let us denote by $X_n$ the set of the first $n$ vertices of $R$.

(ii)$\Rightarrow$(i) 
By assumption, $\{\, \Omega(X_n,\omega) \colon n\in\N\,\}$ is a countable neighbourhood base for $\omega \in \Omega(G)$.
Then $f_\omega$ is countably determined by the directional choices $(X_n,f_\omega(X_n))$ because the end space is Hausdorff.

(i)$\Rightarrow$(ii)
We claim that $\{\,\Omega(X_n,\omega)\colon n\in\N\,\}$ is a neighbourhood base for \mbox{$\omega\in\Omega(G)$}. 
Now, suppose for a contradiction that there is  a basic open neighbourhood $\Omega(X,\omega)$ of $\omega$ in $\Omega(G)$ that contains none of the sets $\Omega(X_n,\omega)$. 
We recursively construct a sequence of pairwise disjoint rays $R_n$ all having precisely their first vertex on $R$ and belonging to ends not in $\Omega(X, \omega)$. 
Having these rays at hand will give the desired contradiction; as $X$ is finite one of these rays lies in $C(X,\omega)$ contradicting that its end is not in $\Omega(X,\omega)$. 

So suppose we have found $R_0,\ldots, R_{n-1}$.
In order to define $R_n$, choose $k \in \N$ large enough that $(X_k,f_\omega(X_k))$ distinguishes $f_\omega$ from all directions induced by the rays $R_0,\ldots, R_{n-1}$ (if $n=0$ pick $k=0$).
Such a $k$ exists because $R$ is directional.
Since the rays $R_0, \ldots ,R_{n-1}$ have precisely their first vertex on $R$ and $X_k$ consists of vertices of $R$, none of the rays $R_0, \ldots ,R_{n-1}$ meets the component $f_\omega(X_k)=C(X_k,\omega)$.
 By our assumption there is an end, $\eta$ say, that is contained in $\Omega(X_k,\omega)$ but not  in $\Omega(X,\omega)$. 
We choose any ray of $\eta$ in $C(X_k,\omega)$ having precisely its first vertex on $R$ to be the $n$th ray $R_n$.
\end{proof}

\begin{theorem}\label{thm: equiv. first countable}
For every graph $G$ and every end $\omega$ of $G$ the following assertions are equivalent:

\begin{enumerate}
    \item The end space of $G$ is first countable at $\omega$.
    \item The direction $f_\omega$ is countably determined in $G$.
    \item The end $\omega$ is represented by a directional ray.
    \item The end $\omega$ is represented by a topological ray.
\end{enumerate}
\end{theorem}

\noindent This theorem clearly implies Theorem~\ref{introTheoremOne}:

\begin{proof}[{Proof of Theorem~\ref{introTheoremOne}}]
Theorem~\ref{thm: equiv. first countable}~(ii)$\Leftrightarrow$(iii) is the statement of Theorem~\ref{introTheoremOne}.
\end{proof}

\begin{proof}
(i)$\Rightarrow$(ii)
Let $\{\, \Omega(X_n,\omega) \colon n\in\N\,\}$ be a countable neighbourhood base of basic open sets for $\omega \in \Omega(G)$.
Then $f_\omega$ is countably determined by its countably many directional choices $(X_n,f_\omega(X_n))$ because the end space is Hausdorff.

(ii)$\Rightarrow$(iii) 
Let $\{\, \Omega(X_n,f_\omega(X_n)) \colon n\in\N\,\}$ be a countable set of directional choices that distinguish $f_\omega$ from every other direction.
Fix any ray $R \in \omega$ and denote by $U$ the union of $V(R)$ and all the $X_n$.
By Lemma~\ref{lemma: JungCountableSet} there is a normal tree $T\subseteq G$ that contains $U$ cofinally. 
As $V(R) \subseteq T$ we have that $\omega \in \Abs{T}$. 
Now, by Lemma~\ref{lemma: NormalTreeNormalRay}, there is a normal ray $R_\omega$ in $T$ belonging to $\omega$. 

We claim that $R_\omega$ is directional in~$G$.
For this, it suffices to show that for every other end $\eta \neq \omega$ of $G$ there is a finite initial segment of $R_\omega$ separating $\omega$ and $\eta$.
By assumption, there is $n\in\N$ such that $X_n$ separates $\omega$ and $\eta$.

Let $v$ be any vertex of the ray $R_\omega-\dc{X_n}$ where the down-closure is taken in $T$.
Since $T$ is normal in~$G$, we have $C(\lceil v \rceil, \omega)  \subset C(X_n, \omega)$ by Lemma~\ref{lemma: NTseparationAndComponents}.
In particular, the initial segment $\lceil v \rceil$ of $R_\omega$ separates~$\omega$ from~$\eta$.

(iii)$\Rightarrow$(iv) This is Lemma~\ref{directionalTopological} (i)$\Rightarrow$(ii).

(iv)$\Rightarrow$(i) This holds by the definition of a topological ray.
\end{proof}

Our second main result, the characterisation by unavoidable substructures of the directions of any given graph that are not countably determined in that graph, needs some preparation.

\begin{definition}[Generalised paths]
For a graph $G$ a \emph{generalised path} in~$G$ with \emph{endpoints} $\omega_1\neq \omega_2\in\Omega(G)$ is an ordered pair $(P,\{\omega_1,\omega_2\})$ where $P\subset G$ is one of the following:
\begin{itemize}[label=\textbf{--}]
    \item a double ray with one tail in the end $\omega_1$ and another tail in the end $\omega_2$;
    \item a finite path $v_0\ldots v_k$ such that $v_0$ dominates the end $\omega_1$ and $v_k$ dominates the end $\omega_2$;
    \item a ray in $\omega_1$ whose first vertex dominates $\omega_2$.
\end{itemize}
Two generalised paths $(P,\Psi)$ and $(P',\Psi')$ are \emph{vertex-disjoint} if $P$ and $P'$ are disjoint.
Two generalised paths $(P,\Psi)$ and $(P',\Psi')$ are \emph{disjoint} if they are vertex-disjoint and $\Psi\cap\Psi'=\emptyset$.
\end{definition}

\begin{definition}[Generalised star and sun]
For a graph $G$ a \emph{generalised star} in $G$ with \emph{centre} $\omega\in\Omega(G)$ is a collection of pairwise vertex-disjoint generalised paths $\{\,(P^i,\{\omega,\omega^i\})\colon i\in I\,\}$ such that each end $\omega^i$ is distinct from all other ends $\omega^{j}$ with $j\neq i\in I$. Then the ends $\omega^i$ with $i\in I$ are the \emph{leaves} of the generalised~star.

A generalised star $\{\,(P^i,\{\omega,\omega^i\})\colon i\in I\,\}$ is \emph{proper} if either every path $P^i$ is a double ray or every path $P^i$ is a ray in $\omega^i$ whose first vertex dominates $\omega$.
A proper generalised star in $G$ with centre $\omega$ is also called a \emph{\sun } in $G$ with \emph{centre} $\omega$.
\end{definition}

\noindent In Figures~\ref{fig:SunRay} and~\ref{fig:SunDoubleRay} we see two examples of suns of size eight centred at an end $\omega$. If we increase their size from eight to $\aleph_1$ in the obvious way, then the direction $f_\omega$ is no longer countably determined.



\begin{figure}
\centering
\begin{minipage}{.5\textwidth}
  \centering
  \includegraphics[width=.8\linewidth]{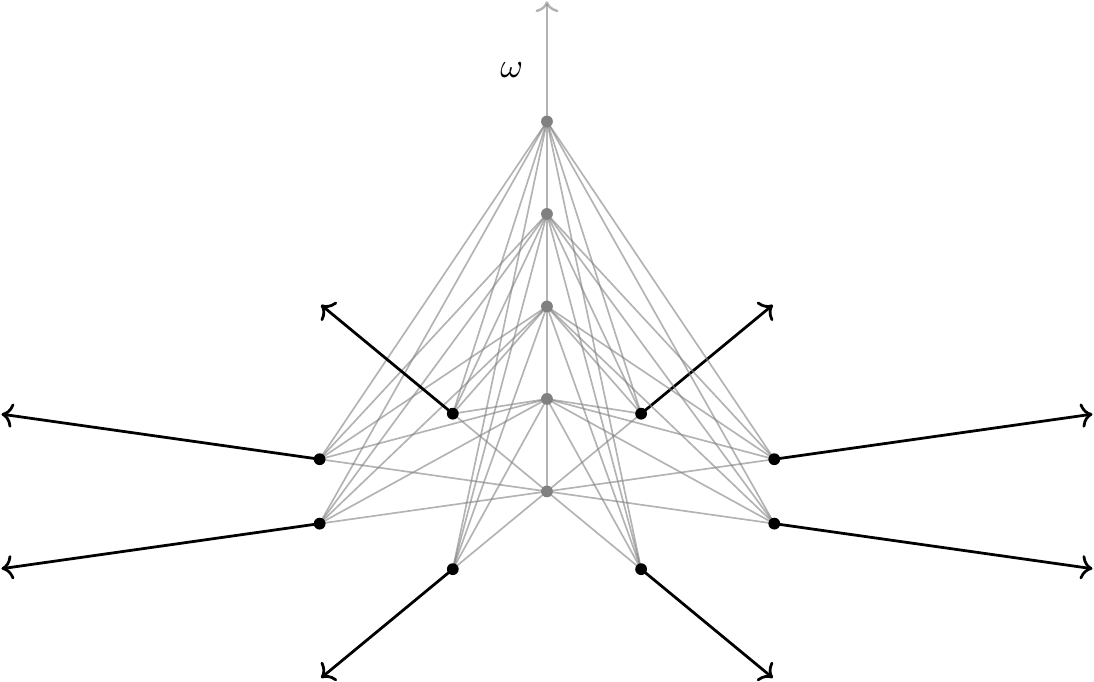}
  \captionof{figure}{The black rays form a sun centred at $\omega$}
  \label{fig:SunRay}
\end{minipage}%
\begin{minipage}{.5\textwidth}
  \centering
  \includegraphics[width=.8\linewidth]{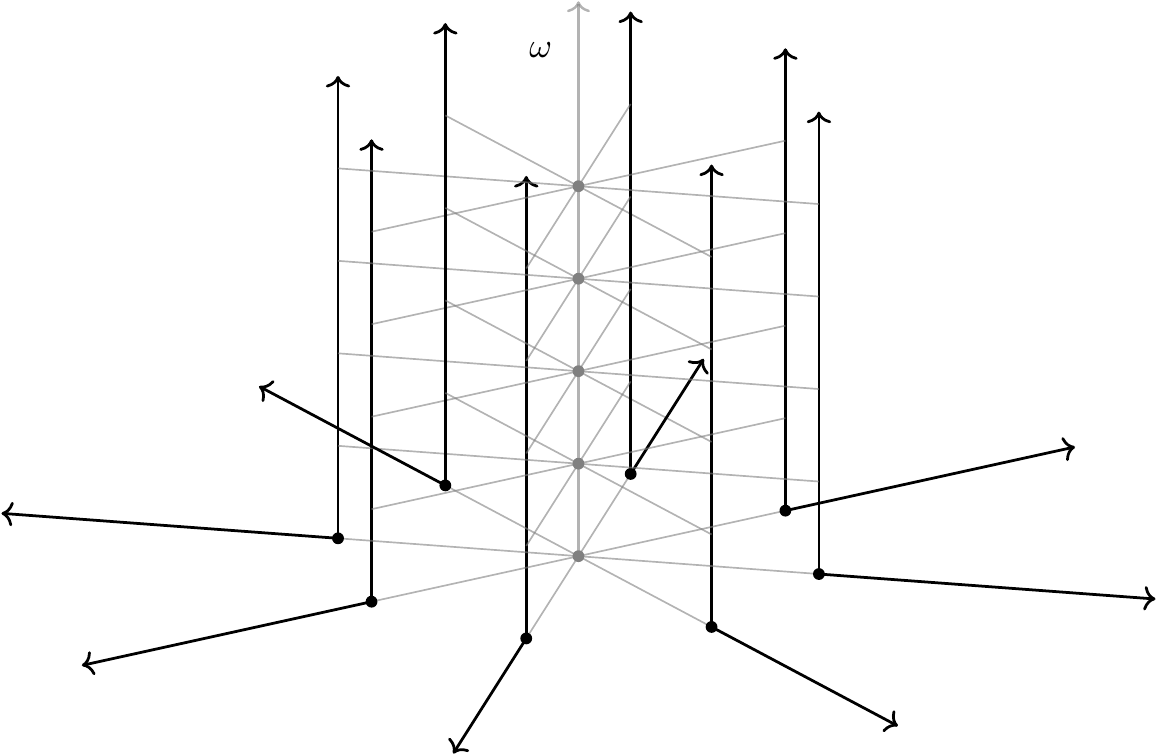}
  \captionof{figure}{The black double rays form a sun centred at $\omega$}
  \label{fig:SunDoubleRay}
\end{minipage}
\end{figure}

\begin{theorem}\label{main:FirstDualStars}
For every graph $G$ and every end $\omega$ of $G$ the following assertions are equivalent:
\begin{enumerate}
    \item The end space of $G$ is not first countable at $\omega$.
    \item There is an uncountable sun in $G$ centred at~$\omega$.
\end{enumerate}

Moreover, if \emph{(ii)} holds, then we find an uncountable sun $\{\,(P^i,\{\omega,\omega^i\})\mid i\in I\,\}$ in $G$ and an uncountable star-decomposition $(T,\cV)$ of $G$ of finite adhesion such that $\omega$ lives in the central part and every $\omega^i$ lives in its own leaf part.
\end{theorem}

\noindent This theorem  
clearly implies Theorem~\ref{introTheoremTwo}:

\begin{proof}[{Proof of Theorem~\ref{introTheoremTwo}}]
Combine Theorem~\ref{main:FirstDualStars} with Theorem~\ref{thm: equiv. first countable} (i)$\Leftrightarrow$(ii).
\end{proof}

\begin{proof}[{Proof of Theorem~\ref{main:FirstDualStars}}] 
(ii)$\Rightarrow$(i)
Suppose for a contradiction that (ii) and $\neg$(i) hold. Let \mbox{$\{(P^i,\{\omega,\omega^i\})\colon i\in I\}$} be an uncountable sun in $G$ centred at $\omega$, and let $\{\, \Omega(X_n,\omega) \colon n\in\N\,\}$ be a countable open neighbourhood base for $\omega$ in $\Omega(G)$.
As all the vertex sets $X_n$ are finite and the $P^i$ are pairwise disjoint, there is an index $i \in I$ such that $P^i$ misses all of the vertex sets $X_n$. Hence, the leaf $\omega^i$ is contained in all of the neighbourhoods $\Omega(X_n, \omega)$ contradicting that $\Omega(G)$ is Hausdorff and that $\{\, \Omega(X_n,\omega) \colon n\in\N\,\}$ is a neighbourhood base for $\omega\in\Omega(G)$.

(i)$\Rightarrow$(ii) Suppose that the end space $\Omega(G)$ is not first countable at~$\omega$.
Our aim is to construct an uncountable sun in $G$ centred at $\omega$. 
Let $\Delta$ be the set of vertices that dominate $\omega$. 
By Zorn's lemma there is an inclusionwise maximal set $\mathcal{R}$ of pairwise disjoint rays all belonging to $\omega$.
Denote by $V[\mathcal{R}]$ the union of all the vertex sets of the rays contained in $\mathcal{R}$, and let $A:= \Delta \cup V[\mathcal{R}]$. 
Then $\Abs{A}=\{\omega\}$.
By Zorn's lemma there is an inclusionwise maximal set $\mathcal{P}$ of pairwise disjoint rays all starting at $A$ and belonging to  ends of $G$ other than $\omega$.

First, note  that $\mathcal{P}$ yields the desired sun if $\mathcal{P}$ is uncountable: 
Since $\Abs{A}=\{\omega\}$ only finitely many rays in $\mathcal{P}$ belong to the same end and every ray in $\mathcal{P}$ has a tail avoiding~$A$. 
Hence, if $\mathcal{P}$ is uncountable, we pass to an uncountable subset $\mathcal{P}^\prime \subseteq \mathcal{P}$ such that  all rays in $\mathcal{P}^\prime$ belong to pairwise distinct ends and have precisely their first vertex in~$A$. 
If uncountably many  rays in $\mathcal{P}^\prime$ start at $\Delta$ we are done. 
So we may assume that uncountably many rays in $\mathcal{P}^\prime$ start at $V[\mathcal{R}]$. 
As only countably many rays in $\mathcal{P}^\prime$ start at the same ray in $\mathcal{R}$, we may pass to an uncountable subset $\mathcal{P}^{\prime \prime} \subseteq \mathcal{P}^\prime$ such that all rays in $\mathcal{P}^{\prime \prime}$ start at distinct rays in~$\mathcal{R}$. Extending every ray in $\mathcal{P}^{\prime \prime}$ by a tail of the unique ray in $\mathcal{R}$ it hits yields again the desired uncountable~sun.

Therefore, we may assume that $\mathcal{P}$ is countable. In the remainder of this proof, we show that this is impossible: we deduce that then there is a countable neighbourhood base for $\omega$ in $\Omega(G) $, contradicting our assumption.

We claim that if $\mathcal{P}$ is finite, then $\omega$ is an isolated point in $\Omega(G)$, that is, there is a finite vertex set separating $\omega$ from all other ends of $G$ simultaneously. Let $\omega_0, \ldots ,\omega_n$ be the ends of the rays in $\mathcal{P}$. As these are only finitely many, there is a finite vertex set $X \subseteq V(G)$ separating $\omega$ from all of the $\omega_i$ simultaneously.  
Now, $C(X, \omega)$ contains only finitely many vertices of the rays in $\mathcal{P}$; by possibly extending $X$ we may assume that $C(X,\omega)$ contains no vertex from any ray in $\mathcal{P}$. 
Then no end of $G$ other than $\omega$ lies in $\Omega(X,\omega)$, because any such end has a ray in $C(X,\omega)$
 
starting at $A$ and avoiding all rays in $\cP$, contradicting the maximality of $\mathcal{P}$.

Thus, $\mathcal{P}$ must be countably infinite. Then the vertex set $V[\mathcal{P}] = \bigcup_{R \in \cP} V(R)$ is countable as well. 
By Lemma~\ref{lemma: JungCountableSet} there is a normal tree $T\subset G$ that contains $V[\mathcal{P}]$ cofinally. 
Moreover, as $\mathcal{P}$ is infinite we have $\omega \in \Abs{T}$, and so there is a normal ray $R_\omega\subset T$ belonging to $\omega$ by Lemma~\ref{lemma: NormalTreeNormalRay}.

We claim that for any end $\eta \neq \omega$ of $G$ there is a finite initial segment of $R_\omega$ separating $\omega$ from $\eta$ in~$G$.
This suffices to derive the desired contradiction, because then Lemma~\ref{directionalTopological} shows that the finite initial segments of $R_\omega$ define a countable open neighbourhood base for $\omega$ in~$\Omega(G)$.

First, suppose $\eta \in \Abs{T}$. Then $\eta$ has a normal ray $R_\eta$ in $T$ by Lemma~\ref{lemma: NormalTreeNormalRay}. 
As $T$ is normal in $G$, the initial segment $R_\omega \cap R_\eta$ of $R_\omega$ separates $\omega$ from $\eta$ in $G$.

Second, suppose $\eta \not\in \Abs{T}$. Then there is a unique component $C$ of $G-T$ that contains a tail of every ray in $\eta$. 
The neighbourhood $N(C)$ of $C$ in $T$ is a chain. 
If~the neighbourhood $N(C)$ of $C$ is not cofinal in~$R_\omega$, then any finite initial segment of $R_\omega$ containing $N(C) \cap R_\omega$  separates $\omega$ from $\eta$ in $G$. So suppose that $N(C)$ is cofinal in $R_\omega$ and denote by $U$ the set of all the vertices in $C$ having a neighbour in $R_\omega$. 

If some vertex $u\in U$ sends infinitely many edges to $T$, then $u$ dominates $\omega$; in particular, there is a generalised path $(P,\{\eta,\omega\})$ in $G$ where $P$ is a ray in $\eta$ that is contained in $C$ and starts at $u\in\Delta\subset A$, contradicting the maximality of~$\mathcal{P}$.
Therefore, we may assume that every vertex in $U$ sends only finitely many edges to~$T$; in particular, $U$ is infinite.
Thus, we find an independent set $M$ of infinitely many $U$--$T$ edges in $G$; we denote by $U'$ the set of endvertices that these edges have in~$U$. Then we apply the star-comb lemma~(\ref{lemma: star-comb}) to $U'$ in~$C$. 
If this yields a star \at $U'$, then the centre of the star dominates $\omega$ and we obtain the same contradiction as above. 
Otherwise this yields a comb \at $U'$. 
Then its spine, $R$ say, is a ray belonging to $\omega$. 
Thus, by the maximality of $\mathcal{R}$ there is a vertex of $V[\mathcal{R}]$ on $R$ and in particular in $C$. Consequently, there is a ray in  $\eta$ that is contained in $C$ and starts at $V[\mathcal{R}]\subset A$, contradicting the maximality of~$\mathcal{P}$.
This completes the proof of (i)$\Rightarrow$(ii).

Finally, we prove the `moreover' part.
Let $S=\{\,(P^i,\{\omega,\omega^i\})\colon i\in I\,\}$ be any uncountable \sun\ in $G$ with centre $\omega$.
We say that an oriented finite-order separation $(A,B)$ of $G$ is $S$-\emph{separating} if $(A,B)$ separates the centre  of $S$ from some leaf $\omega^i$ of $S$ in that $C(A\cap B,\omega)\subset G[ B \setminus A ]$ while \hbox{$C(A\cap B,\omega^i)\subset G[A \setminus B ]$.}
Oriented separations of the form $\lsep{C}{N(C)}$ with $C=C(X,\omega')$ for some finite vertex set $X\in\cX$ and an end $\omega'$ of $G$ are called \emph{golden}.
A star $\sigma$ of finite-order separations is \emph{golden} if every separation in $\sigma$ is golden.

Consider the set $\Sigma$ of all the golden stars that are formed by $S$-separating separations of~$G$, partially ordered by inclusion, and apply Zorn's lemma to $(\Sigma,\subset)$ to obtain a maximal element $\sigma\in \Sigma$.
We claim that $\sigma$ must be uncountable, and assume for a contradiction that $\sigma$ is countable.
Let us write $U$ for the union of the separators of the separations in $\sigma$.
As $U$ is countable, some path $P^j$ avoids $U$.
We consider the two cases that the end $\omega^j$ lies in the closure of $U$ or not.

First, suppose that the end $\omega^j$ does not lie in the closure of $U$. It is straightforward to find an $S$-separating golden separation $\lsep{C}{X}$ with $\omega^j$ living in $C$ and $C$ avoiding $U$.
Note that $\lsep{C}{X}$ is not contained in $\sigma$.
We claim that $\sigma':=\sigma \cup \{\,\lsep{C}{X}\,\}$ is again a star contained in~$\Sigma$.
Since all the elements of $\sigma'$ are $S$-separating and golden, it remains to show that the separations in $\sigma'$ indeed form a star.
As $\sigma\subset\sigma'$ already is a star, it suffices to show $\lsep{C}{X}\le\rsep{Y}{D}$ for all separations $\lsep{D}{Y}\in\sigma$.
For this, let any separation $\rsep{D}{Y}\in\sigma$ be given.
To establish $\lsep{C}{X}\le\rsep{Y}{D}$ it suffices to show that $C$ avoids $Y\cup D$, because $X$ is the neighbourhood of $C$ and $Y$ is the neighbourhood of $D$.
The component $C$ avoids $Y$ because it avoids $U$ which contains~$Y$ as a subset. Therefore, the component $C$ is contained in some component of $G-Y$. 
Now suppose for a contradiction that $C$ and $D$ meet.
Then $C$ must be contained in $D$.
Since $P^j$ avoids $Y$ and contains a ray that lies $C$, we deduce $P^j\subset D$.
But then $\omega$ must live in $D$, contradicting that $(D,Y)$ is $S$-separating.
Thus, $\sigma'$ is again an element of $\Sigma$, contradicting the maximal choice of $\sigma$.

Second, suppose that the end $\omega^j$ lies in the closure of $U$.
We show that this implies $\omega^j=\omega$, a contradiction.
For this, let any finite vertex set $X\subset V(G)$ be given; we show $C(X,\omega^j)=C(X,\omega)$.
To get started, we observe that all but finitely many of the paths $P^i$ avoid $X$.
Also, all but finitely many of the separations $\lsep{D}{Y}\in\sigma$ have their component $D$ avoid $X$ (the components are disjoint for distinct separations in $\sigma$ because $\sigma$ is a star of separations).
As $U$ meets $C(X,\omega^j)$ infinitely, this allows us to find a separation $\lsep{C(Y,\omega^i)}{Y}\in\sigma$ that has its separator $Y$ meet the component $C(X,\omega^j)$ while both $C(Y,\omega^i)$ and $P^i$ avoid the finite separator $X$.
To show $C(X,\omega^j)=C(X,\omega)$, it suffices to find a $P^i$--$P^j$ path in $G$ that avoids~$X$.
We find such a path as follows:
In $C(Y,\omega^i)$ we find a path from $P^i$ to a vertex that sends an edge to a vertex $v$ in the non-empty intersection $Y\cap C(X,\omega^j)$.
And in $C(X,\omega^j)$ we find a path from $P^j$ to $v$.
Then both paths avoid $X$, and their union contains the desired $P^i$--$P^j$ path avoiding $X$.
This concludes the proof that $\sigma$ is uncountable.

Now $\sigma$ is an uncountable star of separations~$\lsep{C}{N(C)}$ such that some leaf $\omega^{i(C)}$ of the sun $S$ lives in $C$ and $\omega$ does not live in~$C$.
Then $J:=\{\,i(C)\colon \lsep{C}{N(C)}\in\sigma\,\}$ is an uncountable subset of~$I$.
In particular, the uncountable sun \[\{\,(P^j,\{\omega,\omega^j\})\colon j\in J\,\}\] and the star-decomposition arising from $\sigma$ are as desired.
\end{proof}

For the interested reader we remark that, even though end spaces of graphs are in general not first countable, it is straightforward to show that every end space is strong Fréchet–Urysohn (which is a generalisation of the first axiom of countability):
A topological space $X$ is called a \emph{strong Fréchet–Urysohn space} if for any sequence of subsets $A_0, A_1, \ldots$ of $X$ and every $x \in \bigcap_{n \in \mathbb{N}} \overline{A_n}$ there is a sequence of points $x_0,x_1,\ldots$ converging to $x$ such that $x_n \in A_n$ for all $n\in\N$.

\begin{lemma}
End spaces of graphs are strong Fréchet–Urysohn.\qed
\end{lemma}

\section{Countably determined graphs and the second axiom\texorpdfstring{\\}{ }of countability}\label{section: second countable}

\noindent In this section we structurally characterise both the graphs that are countably determined and the graphs that are not countably determined.

Clearly, a graph $G$ is countably determined if and only if every component of $G$ is countably determined and only countably many components of $G$ have directions.
Similarly, the end space of a graph $G$ is second countable if and only if every component of $G$ has a second countable end space and only countably many components of $G$ have ends.
Thus, to structurally characterise the countably determined graphs and the graphs that are not countably determined, and to link this to whether or not the end space is second countable, it suffices to consider only connected graphs.

The local property that every direction of $G$ is countably determined in $G$ does not imply the stronger global property that $G$ is countably determined:

\begin{example}\label{countablyDeterminedLocalNotGlobal}
There exists a connected graph $G$ all whose directions are countably determined in it but which is itself not countably determined.
The graph $G$ can be chosen so that its end space is compact and first countable at every end, but neither metrisable nor second countable nor separable.
\end{example}

Recall that a topological space is called \emph{separable} if it admits a countable dense subset.
Every second countable space is separable, but the converse is generally false.
For end spaces, however, we shall see in Theorem~\ref{EquivalentForSecondcount} that the converse is true: the end space of any graph is second countable if and only if it is separable.

\begin{proof}[{Proof of Example~\ref{countablyDeterminedLocalNotGlobal}}]
Let $T_2$ be the rooted infinite binary tree.
The graph $G$ arises from $T_2$ by disjointly adding a new ray $R'$ for every rooted ray $R\subset T_2$ such that $R'_1$ and $R_2'$ are disjoint for distinct rooted rays $R_1,R_2\subset T_2$, and joining the first vertex $v_{R'}$ of each $R'$ to all the vertices of $R$.
Then for every rooted ray $R\subset T_2$ the two rays $R'$ and $v_{R'}R$ are directional in $G$.
Since every direction of $G$ is induced by precisely one of these directional rays, all the directions of $G$ are countably determined.

The graph $G$, however, is not countably determined:
If $\{\,(X_n,C_n)\colon n\in\N\,\}$ is any countable collection of directional choices in $G$, then there is a rooted ray $R\subset T_2$ such that $R'$ avoids all $X_n$ (because $T_2$ contains continuum many rooted rays and $\bigcup_n X_n$ is countable).
But then, for all $n\in\N$, the subgraph $G-X_n$ contains a double ray formed by $R'$ and a subray of $R$ avoiding $X_n$ (that is connected to $R'$ by one of the infinitely many $v_{R'}$--$R$ edges).
These double rays then witness that none of the directional choices $(X_n,C_n)$ distinguishes the direction induced by $R$ from the direction induced by $R'$ or vice versa.

The end space of $G$ is first countable at every end because every direction of $G$ is countably determined (Theorem~\ref{thm: equiv. first countable}).
It is compact because the deletion of any finite set of vertices of $G$ leaves only finitely many components, cf.~\cite[Theorem~4.1]{VTopComp} or~\cite[Lemma~2.8]{StarComb1StarsAndCombs}.
However, the end space of $G$ is not separable, because every dense subset of $\Omega(G)$ must contain all the continuum many ends represented by the rays~$R'$.
Thus, it its neither second countable nor metrizable.
\end{proof}

This is essentially a combinatorially constructed version of the Alexandroff double circle~\cite[Example~3.1.26]{EngelkingBook}.

Now we structurally characterise the countably determined graphs and the graphs that are not countably determined, and structurally characterise the graphs whose end spaces are second countable or not.
Our introduction suggests that this is the order in which we prove these results, but we will prove them in a different order:
First, we shall structurally characterise the graphs whose end spaces are second countable or not.
Then, we shall prove that the end space of a graph is second countable if and only if the graph is countably determined.
Finally, we shall use this equivalence to immediately obtain structural characterisations of the countably determined graphs and the graphs that are not countably determined.
Here, then is our structural characterisation of the graphs whose end spaces are second countable:

\begin{theorem}\label{EquivalentForSecondcount}
For every connected graph $G$ the following assertions are equivalent:
\begin{enumerate}
    \item The end space of $G$ is second countable.
    \item The end space of $G$ is separable.
    \item There is a countable end-faithful normal tree $T\subset G$.
    \item The end space of $G$ has a countable base that consists of basic open sets.
\end{enumerate}
\end{theorem}
\begin{proof}

(i)$\Rightarrow$(ii) Every second countable space is separable.

(ii)$\Rightarrow$(iii) Let $\Psi \subseteq \Omega(G) $ be any countable and dense subset. 
Pick a ray $R_\omega \in \omega$ for every end $\omega \in \Psi$ and let $U:= \bigcup\,\{\, V(R_\omega)\colon\omega\in\Psi\,\}$. 
By Lemma~\ref{lemma: JungCountableSet} there is a countable normal tree $T \subseteq G$ that contains $U$ cofinally. 
We have to show that $T$ is end-faithful. 
As mentioned in Section~\ref{setion: preliminaries}, normal trees are always end-injective. 
To show that $T$ is end-surjective note that $\Abs{T}$ is closed in $\Omega (G)$. 
Hence we have $\Omega(G) = \overline{\Psi} \subseteq \overline{ \Abs{T} }  = \Abs{T}$. So by Lemma~\ref{lemma: NormalTreeNormalRay} the normal tree
$T$ contains a normal ray of every end of $G$.

(iii)$\Rightarrow$(iv)
Let $T\subset G$ be any countable end-faithful normal tree.
We claim that the collection $\mathcal{B} := \{\, \Omega(\lceil t \rceil, \omega)  \colon t\in T,\, \omega \in \Omega \,\}$ is a countable base of the topology on $\Omega (G)$. Note first that $\mathcal{B}$ is indeed countable: 
Since $T$ is end-faithful and countable, the deletion of finitely many vertices of $T$ from $G$ results in only countably many components containing an end. Consequently, for every $t \in T$ there are only countably many distinct sets of the form $\Omega(\lceil t \rceil, \omega)$.

Now, given a basic open set of $\Omega (G)$, say $\Omega(X,\omega)$, our goal is to find a vertex $t \in T$ such that $\Omega (\lceil t \rceil, \omega ) \subseteq \Omega(X,\omega)$. 
By Lemma~\ref{lemma: NormalTreeNormalRay}, every end $\eta$ of $G$ in the closure of $T$ contains a normal ray $R_\eta\subset T$.  
By the normality of $T$ and Lemma~\ref{lemma: NTseparationAndComponents}, every end $\eta\neq\omega$ of $G$ is separated from $\omega$ in $G$ by the finite initial segment $R_\eta\cap R_\omega$ of~$R_\omega$.
In particular, $R_\omega$ is directional in $G$.
Hence, by the implication (i)$\Rightarrow$(ii) of Lemma~\ref{directionalTopological} the ray $R_\omega$ is topological. 
Thus, there is a vertex $t \in R_\omega$ such that $\Omega( \lceil t \rceil ,\omega ) \subseteq \Omega(X,\omega)$ holds.

(iv)$\Rightarrow$(i) This is clear.
\end{proof}

Next, we structurally characterise the graphs whose end spaces are not second countable.
The characterising structure is not the star-decomposition in Theorem~\ref{introTheoremFour} that one would expect; that is because this result is an auxiliary result that we will use in a second step to prove a second structural characterisation, Theorem~\ref{main:SecondDualGoldenStar}, which then is phrased in terms of the desired star-decomposition.


\begin{theorem}\label{main:SecondDualThree}
For every connected graph $G$ the following assertions are equivalent:
\begin{enumerate}
    \item{The end space of $G$ is not second countable.}
    \item{The graph $G$ contains either 
        \begin{itemize}[label=\textbf{--}]
        \item{an uncountable sun, }
        \item{uncountably many disjoint generalised paths, or}
        \item{a finite vertex set that separates uncountably many ends of $G$ simultaneously.}
    \end{itemize} }
\end{enumerate}
\end{theorem}

\begin{proof}
Recall that, by Theorem~\ref{EquivalentForSecondcount}, the end space of $G$ is second countable if and only if it has a countable base that consists of basic open sets.
Then clearly (ii)$\Rightarrow$(i). 

(i)$\Rightarrow$(ii)
For this, suppose that $G$ is given such that the end space of $G$ is not second countable. We have to find one of the three substructures for $G$ listed in~(ii). By Zorn's lemma we find an inclusionwise maximal collection $\mathcal{P}$ of pairwise vertex-disjoint generalised paths in $G$. Our proof consists of two halves.
In the first half we show that if $\mathcal{P}$ is uncountable, then we find either an uncountable sun in $G$ or uncountably many disjoint generalised paths in $G$.
In the second half we show that if $\mathcal{P}$ is countable, then we find a finite vertex set of $G$ that separates uncountably many ends of $G$ simultaneously.

First, we assume that $\mathcal{P}$ is uncountable. In this case, we consider the auxiliary multigraph that is defined on the set of ends of $G$ by declaring every generalised path $(P,\{\omega_1,\omega_2\})\in\mathcal{P}$ to be an edge between $\omega_1$ and~$\omega_2$.
Note that the auxiliary multigraph contains only finitely many parallel edges between any two vertices.
Thus, by replacing $\mathcal{P}$ with a suitable uncountable subset we may assume that the auxiliary multigraph is in fact a graph.

If that auxiliary graph has a vertex $\omega$ of uncountable degree, then its incident edges correspond to uncountably many generalised paths that form an uncountable generalised star in $G$ with centre $\omega$.
This uncountable generalised star need not be proper in general. 
However, it shows that $\omega$ has no countable neighbourhood base in $\Omega (G)$, so Theorem~\ref{main:FirstDualStars} yields an uncountable sun in $G$ with centre~$\omega$.

Otherwise, every vertex of the auxiliary graph has countable degree. Then we greedily find an uncountable independent edge set, and this edge set corresponds to an uncountable collection of disjoint generalised paths in $G$.

Second, we assume that $\mathcal{P}$ is countable. 
Then our goal is to find a finite vertex set $X\subset V(G)$ that separates uncountably many ends of $G$ simultaneously. By Lemma~\ref{lemma: JungCountableSet} there is a countable normal tree $T\subset G$ that cofinally contains the union of the vertex sets of the generalised paths in $\mathcal{P}$. 
Then $\Omega (G) \setminus \Abs{T}$ is uncountable, since otherwise applying Lemma~\ref{lemma: JungCountableSet} to the union of $V(T)$ with the vertex set of a ray from every end in $\Omega(G) \setminus \Abs{T}$ gives a countable  end-faithful normal tree in $G$, contradicting Theorem~\ref{EquivalentForSecondcount}.

Every ray from an end in $\Omega(G) \setminus \Abs{T}$ has a tail in one of the components of $G-T$ and this component is the same for any two rays in the same end. We say that an end in $\omega \in \Omega(G) \setminus \Abs{T}$ \emph{lives} in the unique component of $G-T$ in which every ray in $\omega$ has a tail. 
By the maximality of $\mathcal{P}$, distinct ends in $\Omega(G) \setminus \Abs{T}$ live in distinct components of $G-T$. As $\Omega(G) \setminus \Abs{T}$ is uncountable, we conclude that there are uncountably many components of $G-T$ in which an end of $\Omega(G) \setminus \Abs{T}$ lives; we call these components \emph{good}.

We claim that every good component of $G- T$  has finite neighbourhood. 
For this, assume for a contradiction that there is a good component $C$ of $G-T$ whose neighbourhood $N(C)\subset T$ is infinite. 
Write $\omega$ for the end in $\Omega(G) \setminus\Abs{T}$ that lives in~$C$.
The down-closure of $N(C)$ in $T$ forms a ray and we denote by $\eta$ the end in $\Abs{T}$ represented by this ray.
Consider the set $U$ of all the vertices in $C$ sending an edge to $T$. If some vertex $u\in U$ sends infinitely many edges to $T$, then $u$ dominates $\eta$; in particular, there is a generalised path $(P,\{\omega,\eta\})$ in $G$ where $P$ is a ray in $\omega$ that is contained in $C$ and starts at $u$, contradicting the maximality of $\mathcal{P}$. 
Therefore, we may assume that every vertex in $U$ sends only finitely many edges to~$T$; in particular, $U$ is infinite. Thus, we find an independent set $M$ of infinitely many $U$--$T$ edges in $G$; we denote by $U'$ the set of the endvertices that these edges have in~$U$. Applying the star-comb lemma \ref{lemma: star-comb} in $C$ to $U'$ gives either a star attached to $U'$ or a comb attached to $U'$. 
The centre of a star \at $U'$ would dominate~$\eta$, yielding the same contradiction that would be caused by a vertex in $U$ sending infinitely many edges to $T$.
Hence we obtain a comb attached to $U'$. 
The comb's spine represents $\eta$, because of the edges in $M$.
Consequently,  there is a double ray $P\subset C$ defining a generalised path $(P,\{\omega,\eta\})$ vertex-disjoint from all generalised paths in $\mathcal{P}$, contradicting the maximality of $\mathcal{P}$. This completes the proof of the claim that every good component of $G-T$ has finite neighbourhood.

Finally, as all of the uncountably many good components of $G-T$ have a finite neighbourhood in $T$ and $T$ is countable, there are uncountably many such components having the same finite neighbourhood $X\subset V(T)$. 
Then $X$ is a finite vertex set of $G$ that separates uncountably many ends of $G$ simultaneously, as desired.
\end{proof}

Next, we will prove the structural characterisation of the graphs whose end spaces are not second countable, in terms of the desired star-decomposition:

\begin{theorem}\label{main:SecondDualGoldenStar}
For every connected graph $G$ the following assertions are equivalent:
\begin{enumerate}
    \item The end space of $G$ is not second countable.
    \item $G$ has an uncountable star-decomposition of finite adhesion such that in every leaf part there lives an end of~$G$.
    In particular, the end space of $G$ contains uncountably many pairwise disjoint open~sets.
\end{enumerate}
\end{theorem}

Curiously, the star-decomposition in (ii) cannot be replaced with a star of pairwise inequivalent rays:

\begin{example}
There is a graph $G$ with second countable end space that contains uncountably many rays meeting precisely in their first vertex and representing distinct ends of $G$.
\end{example}
\begin{proof}
We start with $T_2$, the rooted binary tree.
Every end $\omega\in\Omega(T_2)$ is represented by a ray $R_\omega$ starting at the root.
We obtain the graph $H$ from $T_2$ by adding, for every end $\omega\in\Omega(T_2)$, a new ray $R'_\omega$ and joining the $n$th vertex of $R'_\omega$ to the $n$th vertex of $R_\omega$.
Note that the natural inclusion $\Omega(T_2)\subset\Omega(H)$ is a homeomorphism, so $\Omega(H)$ is second countable.
Finally, we obtain the graph $G$ from $H$ by adding a single new vertex $v^\ast$ that we join to the first vertex of every ray~$R'_\omega$.
The addition of $v^\ast$ did not affect the end space.
However, now $v^\ast$ together with its incident edges and all the rays $R'_\omega$ gives the desired substructure.
\end{proof}

We need the next lemma for the proof of Theorem~\ref{main:SecondDualGoldenStar}.
Recall that oriented separations of the form $\lsep{C}{N(C)}$ with $C=C(X,\omega)$ for some finite vertex set $X\in\cX$ and an end $\omega$ of $G$ are called \emph{golden}.
A star $\sigma$ of finite-order separations is \emph{golden} if every separation in $\sigma$ is golden.

\begin{lemma}\label{matchingToGoldenStar}
Let $G$ be any connected graph.
If there exist uncountably many pairwise vertex-disjoint generalised paths in~$G$, then $G$ admits an uncountable golden star of separations.
\end{lemma}

\begin{proof} 
Let $\{\,(P^i,\{\omega_1^i,\omega_2^i\})\colon i\in I\,\}$ be any uncountable collection of pairwise vertex-disjoint generalised paths in $G$.
By the pigeonhole principle there exists a number $k\in\N$ and an uncountable subset $J\subset I$ such that for all $j \in J$ the ends $\omega_1^j$ and $\omega_2^j$ are $k$-distinguishable.
Without loss of generality $J=I$.
By Theorem~\ref{magicTDC} we find a tree-decomposition $(T,\cV)$ of $G$ that efficiently distinguishes all the \mbox{$k$-dis}tinguishable ends of~$G$.

Fix an arbitrary root $r\in T$ and write $F$ for the collection of all the edges $e\in T$ whose induced separation distinguishes two ends $\omega_1^i$ and $\omega_2^i$.
Then let $T'\subset T$ be the subtree that is induced by the down-closure of the endvertices of the edges in $F$ in the rooted tree $T$.
If $T'$ has a vertex $t$ of uncountable degree, then we find an uncountable subset $\Psi\subset\{\,\omega_1^i,\omega_2^i\colon i\in I\,\}$ such that every end $\omega\in\Psi$ lives in its own component $C_\omega$ of $G-V_t$ with finite neighbourhood; in particular, $\{\,\lsep{C_\omega}{N(C_\omega)}\colon \omega\in\Psi\,\}$ is the desired uncountable golden star of separations.
We claim that $T'$ must have a vertex of uncountable degree.
Otherwise, $T'$~is countable.
Then the union $U$ of the separators of the separations induced by the edges of $T'\subset T$ is a countable vertex set. In order to obtain a contradiction note that every $P^i$ meets $U$ in at least one vertex and these vertices are distinct for distinct $P^i$.
\end{proof}

\begin{proof}[{Proof of Theorem~\ref{main:SecondDualGoldenStar}}]
Recall that, by Theorem~\ref{EquivalentForSecondcount}, the end space of $G$ is second countable if and only if it has a countable base that consists of basic open sets.
Then clearly (ii)$\Rightarrow$(i).

(i)$\Rightarrow$(ii) Suppose that the end space of $G$ is not second countable. We are  done if there is a finite vertex set separating uncountably many ends of $G$ simultaneously. Therefore, we may assume by Theorem~\ref{main:SecondDualThree} that either there is an uncountable sun in $G$ or $G$ contains uncountably many disjoint generalised paths. In either case we are done by Lemma~\ref{matchingToGoldenStar}.
\end{proof}

Theorems~\ref{EquivalentForSecondcount} and~\ref{main:SecondDualGoldenStar} structurally characterise the graphs whose end spaces are second countable or not, by the structures in terms of which Theorems~\ref{introTheoremThree} and~\ref{introTheoremFour} are phrased.
The next theorem allows us to deduce Theorems~\ref{introTheoremThree} and~\ref{introTheoremFour} from Theorems~\ref{EquivalentForSecondcount} and~\ref{main:SecondDualGoldenStar} immediately.

\begin{theorem}\label{lemma: G count. det. iff Omega second count.}
The end space of a graph is second countable if and only if the graph is countably determined.
\end{theorem}
\begin{proof}
Let $G$ be any graph.
For the forward implication, suppose that the end space of $G$ is second countable.
Then, by Theorem~\ref{EquivalentForSecondcount}, the end space of~$G$ has a countable base $\{\,\Omega(X_n,C_n)\colon n\in\N\,\}$ that consists of basic open sets $\Omega(X_n,C_n)\subset\Omega(G)$.
We claim that the directional choices $(X_n,C_n)$ distinguish every two directions of~$G$ from each other.
For this, let any two distinct directions $f,h$ of $G$ be given.
Since the end space of $G$ is Hausdorff, there is a number $n\in\N$ such that $\Omega(X_n,C_n)$ contains $\omega_f$ but not $\omega_h$; in particular, $f(X_n)=C_n$ and $h(X_n)\neq C_n$ as desired.

For the backward implication suppose that $G$ is countably determined, and let $\{\,(X_n,C_n)\colon n\in\N\,\}$ be any countable set of directional choices $(X_n,C_n)$ in $G$ that distinguish every two directions of $G$ from each other.
Let us assume for a contradiction that the end space of $G$ is not second countable.
Then, by Theorem~\ref{main:SecondDualThree}, the graph $G$ contains either
\begin{itemize}[label=\textbf{--}]
        \item{an uncountable sun, }
        \item{uncountably many disjoint generalised paths, or}
        \item{a finite vertex set that separates uncountably many ends of $G$ simultaneously.}
\end{itemize}
If $G$ contains an uncountable sun or uncountably many disjoint generalised paths, then in either case $G$ contains a generalised path $(P,\{\omega_1,\omega_2\})$ such that $P$ avoids all the countably many finite vertex sets~$X_n$.
But then no directional choice $(X_n,C_n)$ distinguishes $f_{\omega_1}$ from $f_{\omega_2}$ or vice versa, a contradiction.
Thus, there must be a finite vertex set $X\subset V(G)$ that separates uncountably many ends $\omega_i$ ($i\in I$) of $G$ simultaneously.
We abbreviate $C(X,\omega_i)$ as $D_i$.
By the pigeonhole principle we may assume that $X=N(D_i)$ for all $i\in I$.

Let us consider the subset $N\subset\N$ of all indices $n\in\N$ whose directional choice $(X_n,C_n)$ distinguishes some $f_{\omega_i}$ from some $f_{\omega_j}$. Every component $C_n$ with $n\in N$ meets some component $D_i$ because some end $\omega_i$ lives in~$C_n$.
Then, for all $n\in N$, either $C_n$ is contained in some component $D_i$ entirely, or $C_n$ meets $X$ and contains all of the components $D_i$ except possibly for the finitely many components $D_i$ that meet $X_n$.
This means that every directional choice $(X_n,C_n)$ with $n\in N$ either distinguishes finitely many directions $f_{\omega_i}$ from uncountably many directions $f_{\omega_j}$ or vice versa.
Thus, for every $n\in N$ there is a cofinite subset $I_n\subset I$ such that no $f_{\omega_i}$ with $i\in I_n$ is distinguished by $(X_n,C_n)$ from any other $f_{\omega_j}$ with $j\in I_n$.
But then the uncountably many directions $f_{\omega_i}$ with $i\in\bigcap_{n\in N}I_n$ are not distinguished from each other by any directional choices $(X_n,C_n)$, a contradiction.
\end{proof}

\begin{proof}[{Proof of Theorem~\ref{introTheoremThree}}]
Theorem~\ref{EquivalentForSecondcount} and Theorem~\ref{lemma: G count. det. iff Omega second count.} together imply Theorem~\ref{introTheoremThree}.
\end{proof}

\begin{proof}[{Proof of Theorem~\ref{introTheoremFour}}]
Theorem~\ref{main:SecondDualGoldenStar} and Theorem~\ref{lemma: G count. det. iff Omega second count.} together imply Theorem~\ref{introTheoremFour}.
\end{proof}

\section{First and second countability for \texorpdfstring{$\vert G \vert $}{|G|}}\label{section: modG}

\noindent In this section, we employ our results to characterise when the spaces $\vert G\vert$ formed by a graph and its ends are first countable or second countable.

First, we describe a common way to extend the topology on $\Omega(G)$ to a topology on $|G| = G \cup \Omega(G)$. The topology called \textsc{MTop}, has a basis formed by all open sets of $G$ considered as a metric length-space (i.e.\ every edge together with its endvertices forms a unit interval of length~$1$, and the distance between two points of the graph is the length of a shortest arc in $G$ between them), together with basic open neighbourhoods for ends of the form
\begin{align*}
    \hat{C}_\varepsilon(X,\omega) := C(X,\omega) \cup \Omega(X,\omega) \cup \mathring{E}_\varepsilon(X, C(X,\omega)),
\end{align*}
where $\mathring{E}_\varepsilon(X, C(X,\omega))$ denotes the open ball around $C(X,\omega)$ in $G$ of radius $\varepsilon$. Polat observed that the subspace $V(G) \cup \Omega(G)$ is homeomorphic to $\Omega(G^+)$, where $G^+$ denotes the graph obtained from $G$ by gluing a new ray $R_v$ onto each vertex $v$ of $G$ so that $R_v$ meets $G$ precisely in its first vertex $v$ and $R_v$ is distinct from all other $R_{v'}$, cf.~\cite[\S4.16]{polat1996ends}.

Note that $\vert G\vert$ with \textsc{MTop} is first countable at every vertex of $G$ and at inner points of  edges.

\newpage

\begin{lemma}
For every graph $G$ and every end $\omega$ in the space $\vert G\vert$ with \textsc{MTop}, the following assertions are equivalent:
\begin{enumerate}
    \item $\vert G \vert$ is first countable at $\omega$.
    \item The end $\omega$ is represented by a ray $R$ such that every component of $G-R$ has finite neighbourhood.
\end{enumerate}
\end{lemma}
\noindent Note that every ray as in~(ii) is directional and contains all the vertices that dominate the end.

\begin{proof}
(i)$\Rightarrow$(ii) 
If $|G|$ is first countable at $\omega$, then $\Omega(G^+)$ is also first countable at~$\omega$. 
So, by Theorem~\ref{thm: equiv. first countable}, the end $\omega$ of $G$ considered as an end of $G^+$ is represented by a directional ray $R$ in $G^+$. 
We may assume that $R \subseteq G$. 
Now, if there is a component $C$ of $G-R$ such that $N(C)$ is infinite, then $C$ is contained in $C(X,\omega)$ for every finite initial segment $X$ of $R$. Consider a vertex $v \in C$. The end $\omega_v$ of $G^+$ represented by $R_v$ is contained in all of the sets $\Omega_{G^+}(X, \omega)$, contradicting the fact that $R$ is directional.

(ii)$\Rightarrow$(i)Let $R \in\omega$  be as in (ii). First note that for every vertex $v \in V(G)$ there is a finite initial segment $X$ of $R$ that separates $v$ from $\omega$ in that $v \not\in C(X,\omega)$. Now, if $X'\subseteq V(G)$ is a finite vertex set, chose a finite initial segment $X$ of $R$ that separates every vertex in $X'$ from $\omega$. Then $C(X, \omega) \subseteq C(X',\omega)$. Hence the sets $\hat{C}_\frac{1}{n}(X ,\omega)$, with $n\in \N$ and $X$ a finite initial segment of $R$, form a countable neighbourhood base for $\omega$.
\end{proof}

An end is called \emph{fat} if there are uncountably many disjoint rays that represent the end.

\begin{lemma}
For every graph $G$ and every end $\omega$ in the space $\vert G\vert$ with \textsc{MTop}, the following assertions are equivalent:
\begin{enumerate}
    \item $\vert G \vert$ is not first countable at $\omega$.
    
    \item The end $\omega$ is fat or dominated by uncountably many vertices.
\end{enumerate}
\end{lemma}
\begin{proof}
(i)$\Rightarrow$(ii) If $\vert G\vert$ is not first countable at $\omega$, then $\omega$ considered as end of $G^+$ has no countable neighbourhood base in $\Omega(G^+)$ either. By Theorem~\ref{main:FirstDualStars} there is an uncountable sun in $G^+$, 
which gives either uncountably many disjoint rays in $G$ that represent $\omega$ or uncountably many vertices in $G$ that dominate $\omega$.

(ii)$\Rightarrow$(i) Let $\omega$ be an end of $G$. 
If there are uncountably many disjoint rays  that represent $\omega$ or uncountably many vertices  that dominate $\omega$, then any countable collection $\{ \hat{C}_{\varepsilon_n}(X_n,\omega) \mid n \in \N \} $ fails to separate $\omega$ from a vertex on one of these rays or from a dominating vertex.
\end{proof}

Second countability for $\vert G \vert$ is simply decided by the order of $G$:

\begin{lemma}\label{lemma: modG with Mtop second countable}
For every connected graph $G$ and the space  $\vert G \vert$ with \textsc{MTop} the following assertions are equivalent:
\begin{enumerate}
    \item $\vert G \vert$ is second countable.
    \item $V(G)$ is countable.
\end{enumerate}
\end{lemma}
\begin{proof}
(i)$\Rightarrow$(ii) If $G$ has uncountably many vertices, then the uniform stars with radius~$\frac{1}{2}$ around every vertex form an uncountable collection of disjoint open sets. Hence, $\vert G \vert$ is not second countable.

(ii)$\Rightarrow$(i) If $V(G)$ is countable, we take an enumeration $V(G)=\{v_1, v_2, \ldots \}$ of the vertex set of $G$ and write $X_n$ for the set of the first $n$ vertices. For every vertex $v$ the uniform stars with radius $\frac{1}{n}$ around $v$ form a countable neighbourhood  base. There are only countably many edges and each is homeomorphic to the unit interval and therefore has a countable base of its topology. Finally, all the sets of the form  $\hat{C}_{\frac{1}{n}}(X_n, \omega)$ for all $n \in \N$ and all $\omega \in \Omega(G)$ give a countable neighbourhood base for all the ends of~$G$. Note that $G-X_n$ has only countably many components, for all $n\in \N$; therefore there are indeed only countably many sets of the form $\hat{C}_{\frac{1}{n}}(X_n, \omega)$.
\end{proof}

%
%



\bibliographystyle{plain}
\bibliography{reference}

\begin{thebibliography}{10}

\bibitem{StarComb1StarsAndCombs}
C.~Bürger and J.~Kurkofka.
\newblock {Duality theorems for stars and combs I: Arbitrary stars and combs},
  2020.
\newblock \href{https://arxiv.org/abs/2004.00594}{arXiv:2004.00594}.

\bibitem{CanonicalTreeOfTangles}
J.~Carmesin, M.~Hamann, and B.~Miraftab.
\newblock Canonical trees of tree-decompositions, 2020.
\newblock \href{https://arxiv.org/abs/2002.12030}{arXiv:2002.12030}.

\bibitem{diestel2006end}
R.~Diestel.
\newblock End spaces and spanning trees.
\newblock {\em J. Combin.\ Theory (Series B)}, 96(6):846--854, 2006.

\bibitem{diestel2015book}
R.~Diestel.
\newblock {\em {Graph Theory}}.
\newblock Springer, 5th edition, 2015.

\bibitem{EndsAndTangles}
R.~Diestel.
\newblock {Ends and Tangles}.
\newblock {\em {Abh.\ Math.\ Sem.\ Univ.\ Hamburg}}, 87(2):223--244, 2017.
\newblock {Special issue in memory of Rudolf Halin},
  \href{https://arxiv.org/abs/1510.04050}{arXiv:1510.04050v3}.

\bibitem{VTopComp}
R.~Diestel.
\newblock End spaces and spanning trees.
\newblock {\em J. Combin.\ Theory (Series B)}, 96(6):846--854,
  \href{http://www.math.uni-hamburg.de/home/diestel/papers/EndSpacesAndNSTs.pdf}{2006}.

\bibitem{diestelKuhn2003directions}
R.~Diestel and D.~Kühn.
\newblock {Graph-theoretical versus topological ends of graphs}.
\newblock {\em J. Combin.\ Theory (Series B)}, 87(1):197--206, 2003.

\bibitem{EngelkingBook}
R.~Engelking.
\newblock {\em {General Topology}}, volume~6 of {\em Sigma Series in Pure
  Mathematics}.
\newblock Heldermann Verlag, Berlin, second edition, 1989.

\bibitem{Halin_Enden64}
R.~Halin.
\newblock {Über unendliche Wege in Graphen}.
\newblock {\em Math.\ Annalen}, 157:125--137, 1964.

\bibitem{jung1969wurzelbaume}
H.A. Jung.
\newblock {Wurzelb{\"a}ume und unendliche Wege in Graphen}.
\newblock {\em Math.\ Nachr.}, 41:1--22, 1969.

\bibitem{ApproximatingNormalTrees}
J.~Kurkofka, R.~Melcher, and M.~Pitz.
\newblock Approximating infinite graphs by normal trees.
\newblock {\em J.~Combin.\ Theory (Series B)}, 148:173--183, 2021.
\newblock \href{https://arxiv.org/abs/2002.08340}{arXiv:2002.08340}.

\bibitem{EndsTanglesCrit}
J.~Kurkofka and M.~Pitz.
\newblock Ends, tangles and critical vertex sets.
\newblock {\em Math.\ Nachr.}, 292(9):2072--2091, 2019.
\newblock \href{https://arxiv.org/abs/1804.00588}{arXiv:1804.00588}.

\bibitem{StoneCechTangles}
J.~Kurkofka and M.~Pitz.
\newblock {Tangles and the Stone-\v{C}ech compactification of infinite graphs}.
\newblock {\em J.~Combin.\ Theory (Series B)}, 146:34--60, 2021.
\newblock \href{https://arxiv.org/abs/1806.00220}{arXiv:1806.00220}.

\bibitem{polat1996ends}
N.~Polat.
\newblock Ends and multi-endings, {I}.
\newblock {\em J. Combin.\ Theory (Series B)}, 67:86--110, 1996.

\bibitem{polat1996ends2}
N.~Polat.
\newblock Ends and multi-endings, {II}.
\newblock {\em J. Combin.\ Theory (Series B)}, 68:56--86, 1996.

\bibitem{sprussel2008end}
P.~Spr{\"u}ssel.
\newblock End spaces of graphs are normal.
\newblock {\em J. Combin.\ Theory (Series B)}, 98(4):798--804, 2008.

\end{thebibliography}

\end{document}